\documentclass[11pt,reqno]{amsart}

\usepackage{amssymb, latexsym, amsmath}
\usepackage{amsfonts}
\usepackage{amsthm}
\usepackage{enumerate}

\theoremstyle{definition}
\newtheorem{theorem}{Theorem}
\newtheorem{lemma}[theorem]{Lemma}
\newtheorem{proposition}[theorem]{Proposition}
\newtheorem{definition}[theorem]{Definition}
\newtheorem{corollary}[theorem]{Corollary}

\newtheorem{question}[theorem]{Question}

\newtheorem*{mainthm}{Theorem \ref{thm:closed-finite}}

\newcommand\id{\operatorname{id}}

\title[Independence densities of countable hypergraphs]{Positive independence densities of finite rank countable hypergraphs are achieved by finite hypergraphs}

\author[P. Balister]{Paul Balister}
\address{Department of Mathematical Sciences, University of Memphis, Memphis TN 38152, USA}
\email{pbalistr@memphis.edu}
\thanks{Research of the first and second authors supported in part by NSF grant DMS-1301614}

\author[B. Bollob\'as]{B\'ela Bollob\'as}
\address{Department of Pure Mathematics and Mathematical Statistics, University of Cambridge, Wilberforce Road, Cambridge CB3\thinspace0WB, UK; {\em and\/}
Department of Mathematical Sciences, University of Memphis, Memphis TN 38152, USA}
\email{b.bollobas@dpmms.cam.ac.uk}
\thanks{Research of the second author supported in part by EU MULTIPLEX grant 317532}

\author[K. Gunderson]{Karen Gunderson}
\address{Department of Mathematics, University of Manitoba, Winnipeg MB R3T\thinspace2N2, Canada}
\email{karen.gunderson@umanitoba.ca}
\thanks{While the research took place, the third author was employed by the Heilbronn Institute for Mathematical Research, University of Bristol, Bristol, UK}

\date{18 April 2016}


\begin{document}

\begin{abstract}
The independence density of a finite hypergraph is the probability that a subset of vertices, chosen uniformly at random contains no hyperedges.  Independence densities can be generalized to countable hypergraphs using limits.  We show that, in fact, every positive independence density of a countably infinite hypergraph with hyperedges of bounded size is equal to the independence density of some finite hypergraph whose hyperedges are no larger than those in the infinite hypergraph.  This answers a question of Bonato, Brown, Kemkes, and Pra{\l}at about independence densities of graphs.  Furthermore, we show that for any $k$, the set of independence densities of hypergraphs with hyperedges of size at most $k$ is closed and contains no infinite increasing sequences.
\end{abstract}

\maketitle

\section{Introduction}\label{sec:intro}

In a hypergraph, a subset of vertices is said to be \emph{independent} if it contains no hyperedges. Many problems in combinatorics, including questions in extremal graph theory about the number of $H$-free graphs, or Szemer\'edi's theorem on arithmetic progressions among others, can be expressed in terms of the maximum size or number of independent sets in certain hypergraphs.  Recall that the \emph{independence number} of a hypergraph $\mathcal{H}$, denoted $\alpha(\mathcal{H})$, is the maximum size of an independent set in $\mathcal{H}$ and $i(\mathcal{H})$ is the number of independent sets in~$\mathcal{H}$.  Let $\mathcal{I}(\mathcal{H})$ be the set of independent sets in $\mathcal{H}$ so that $|\mathcal{I}(\mathcal{H})| = i(\mathcal{H})$.

There have been a number of results giving bounds on both $\alpha(\mathcal{H})$ and $i(\mathcal{H})$ for certain classes of hypergraphs. Ajtai, Koml\'os, Pintz, Spencer, and Szemer\'edi \cite{AKPSS82} used random methods to give a lower bound on the independence numbers of $k$-uniform hypergraphs with no cycles of length $2$, $3$, or $4$ with a fixed average degree. A hypergraph is said to be \emph{linear} if any pair of hyperedges have at most one common vertex. Duke, Lefmann, and R\"odl \cite{DLR95} gave a similar lower bound for independence numbers of linear hypergraphs with average degree~$t$. Using this result, Cooper, Dutta, and Mubayi \cite{CDM14} gave lower bounds on $i(\mathcal{H})$ for $k$-uniform linear hypergraphs $\mathcal{H}$ with average degree~$t$. Further lower bounds on the independence number of $k$-uniform hypergraphs satisfying a maximum degree condition were given by Kostochka, Mubayi, and Verstra\"ete~\cite{KMV14}. Cutler and Radcliffe \cite{CR13} noted that the Kruskal--Katona theorem implies that the hypergraph on $n$ vertices with $m$ edges and the largest number of independent sets is that given by the first $m$ elements of $\binom{[n]}{k}$ in the lexicographic ordering. They further examined stronger notions of independence and gave asymptotic upper bounds on $i(\mathcal{H})$ for hypergraphs that are dense and have dense complements. Yuster \cite{rY06} considered algorithms for finding independent sets in $k$-uniform hypergraphs. Recent papers by Balogh, Morris, and Samotij \cite{BMS14} and by Saxton and Thomason \cite{ST14} have given new methods for enumerating independent sets in hypergraphs satisfying certain density conditions that have answered a number of previously open problems and provided new proofs of other results.

The \emph{independence density} of a finite hypergraph $\mathcal{H}$, denoted $\id(\mathcal{H})$, is the probability that a set of vertices chosen uniformly at random is an independent set.  That is, if $\mathcal{H}$ is a hypergraph on $n$ vertices, then
\begin{equation}\label{eq:def-id}
 \id(\mathcal{H}) = \frac{i(\mathcal{H})}{2^n}.
\end{equation}
In more generality, for any $p \in (0,1)$, define the \emph{independence $p$-density} of a hypergraph $\mathcal{H}$ by
\begin{equation}\label{eq:def-idp}
\id_p(\mathcal{H}) = \sum_{I \in \mathcal{I}(\mathcal{H})} p^{|I|}(1-p)^{n - |I|}.
\end{equation}
Then, $\id(\mathcal{H})= \id_{\frac{1}{2}}(\mathcal{H})$ and for any $p \in (0,1)$, $\id_p(\mathcal{H})$ is the probability that a subset of vertices, chosen at random with each vertex included independently with probability $p$, is an independent set.  This can be viewed as a re-scaling of the independence polynomial.

The precise definition of independence densities of countable hypergraphs is given in Definition~\ref{def:inf-id} below in terms of limits of independence densities of increasing sequences of finite induced subhypergraphs.

The notion of the independence density of a countable graph in the case $p = 1/2$ was introduced by Bonato, Brown, Kemkes, and Pra{\l}at \cite{BBKP11} who gave bounds on independence densities of graphs in terms of their matching number and showed that the set of independence densities of countable graphs is a set of rationals whose closure is also contained in the rationals. Further, they asked whether the set of positive independence densities of countable graphs is the same as the set of independence densities of finite graphs.

Subsequently, Bonato, Brown, Mitsche, and Pra{\l}at \cite{BBMP14} generalized independence densities to hypergraphs, again in the special case $p = 1/2$, and showed that any hypergraph whose hyperedges are of bounded size has a rational independence density and gave a construction showing that any real number in $[0,1]$ is the independence density of a countable hypergraph whose edge sizes are possibly unbounded.  Indeed, by a straightforward modification of their construction, one can show that for every $p \in (0,1)$ and for every $x \in [0,1]$, there is a countable hypergraph $\mathcal{H}$, whose edge sizes are possibly unbounded, with $\id_p(\mathcal{H}) = x$.

With this in mind, in this paper, only countable hypergraphs with bounded edge size are considered. Recall that the \emph{rank} of a hypergraph $\mathcal{H}$, denoted $r(\mathcal{H})$, is the supremum of the sizes of hyperedges in~$\mathcal{H}$. Following the notation in~\cite{BBMP14}, for any $k \geq 1$ and $p \in (0,1)$, define
\begin{equation}\label{eq:set-of-id}
 \mathbf{H}_{k, p} = \{\id_p(\mathcal{H}) \mid \mathcal{H} \text{ is a countable hypergraph with } r(\mathcal{H}) \leq k\}.
\end{equation}

The main result of this paper is Theorem~\ref{thm:closed-finite} below which shows that, in fact, all positive independence densities in $\mathbf{H}_{k, p}$ are given by finite hypergraphs of  rank at most~$k$. The $k$-uniform hypergraph consisting of countably many mutually disjoint hyperedges has independence $p$-density $0$ and for any $p < 1$, there is no finite rank $k$ hypergraph with independence $p$-density~$0$ since the empty set is always independent.

\begin{theorem}\label{thm:closed-finite}
For every $k \geq 1$ and $p \in (0,1)$, the set $\mathbf{H}_{k,p}$ is closed and furthermore,
\[
 \mathbf{H}_{k,p} = \{0\} \cup \{\id_p(\mathcal{H}) \mid \mathcal{H} \text{ is a finite hypergraph with } r(\mathcal{H}) \leq k\}.
\]
\end{theorem}

In addition, we show in Theorem~\ref{thm:no-inc-seq} that this set of independence densities, $\mathbf{H}_{k,p}$, has no infinite increasing sequences. In the special case of countable graphs and $p = 1/2$, we show, in Proposition~\ref{prop:graphs}, that every non-zero independence density of a countable graph is also the independence density of a finite graph, answering the question of Bonato, Brown, Kemkes, and Pra{\l}at from \cite{BBKP11} in the affirmative.

The remainder of this paper is organized as follows. In Section~\ref{sec:prelim}, some basic facts about independence densities are stated with their proofs. In Section~\ref{sec:results}, the proof of Theorem~\ref{thm:closed-finite} and other results are given.  Finally, in Section~\ref{sec:open}, some open questions are presented.

\section{Preliminaries}\label{sec:prelim}

In this section, some straightforward facts are given on independence densities. The proof that independence densities are well-defined for countable hypergraphs follows exactly as in \cite{BBMP14} for the case $p = 1/2$ and is included here for completeness.

\begin{lemma}\label{lem:sub-hyp}
If $p \in (0,1)$ and $\mathcal{H}$ and $\mathcal{G}$ are finite hypergraphs with $\mathcal{H} \subseteq \mathcal{G}$, then $\id_p(\mathcal{G}) \leq \id_p(\mathcal{H})$.
\end{lemma}
\begin{proof}
Set $V_1 = V(\mathcal{H})$ and $V_2 = V(\mathcal{G}) \setminus V_1$. Since any subset of $V_1$ that is independent in $\mathcal{G}$ is also independent in~$\mathcal{H}$, $\mathcal{I}(\mathcal{G}[V_1]) \subseteq \mathcal{I}(\mathcal{H})$. Further, if $I$ is any independent set in~$\mathcal{G}$, then so are $I \cap V_1$ and $I \cap V_2$. Thus,
\[
\id_p(\mathcal{G}) \leq \id_p(\mathcal{G}[V_1]) \id_p(\mathcal{G}[V_2]) \leq \id_p(\mathcal{H}).
\]
\end{proof}

For any hypergraph $\mathcal{H}$, let $\mu(\mathcal{H})$ be the matching number of~$\mathcal{H}$; the maximum number of vertex-disjoint edges in~$\mathcal{H}$.  The following result is similar to Theorem 3 in \cite{BBMP14}.

\begin{lemma}\label{lem:match-bd}
Let $\mathcal{H}$ be a finite hypergraph with $r(\mathcal{H}) = k$ and a matching of size~$m$. Then,
\[
 \id_p(\mathcal{H}) \leq \big(1 - p^k \big)^m.
\]
\end{lemma}
\begin{proof}
Let $E_1, E_2, \dots, E_m$ be a matching in~$\mathcal{H}$. Any independent set in $\mathcal{H}$ does not contain all of the vertices in any particular edge in the matching. Since the edges of the matching are pairwise disjoint,
\[
 \id_p(\mathcal{H}) \leq \prod_{i = 1}^{m}\big(1-p^{|E_i|}\big) \leq \big(1 - p^k \big)^m.
\]
\end{proof}

The definition of independence density is extended to infinite graphs via limits. The notion of `chains' are used to show that this can be done in an unambiguous way.

\begin{definition}
Let $\mathcal{H}$ be a countable hypergraph. A sequence $\{\mathcal{C}_n\}_{n \geq 1}$ of finite hypergraphs is called a \emph{chain for $\mathcal{H}$} if{f} for every~$n$, $\mathcal{C}_n$ is an induced subhypergraph of $\mathcal{C}_{n+1}$, $\cup_{n \geq 1} V(\mathcal{C}_n) = V(\mathcal{H})$ and $\cup_{n \geq 1} E(\mathcal{C}_n) = E(\mathcal{H})$.
\end{definition}

The next lemma follows the same argument as Theorem 2 in \cite{BBMP14}.

\begin{lemma}
Let $\{\mathcal{C}_n\}_{n \geq 1}$ and $\{\mathcal{G}_n\}_{n \geq 1}$ be chains for a hypergraph~$\mathcal{H}$. Then, for any $p \in (0,1)$,
\[
 \lim_{n \to \infty} \id_p(\mathcal{C}_n) = \lim_{n \to \infty} \id_p(\mathcal{G}_n).
\]
\end{lemma}
\begin{proof}
Note that both of the limits exist since each of $\{\id_p(\mathcal{C}_n)\}_{n \geq 1}$ and $\{\id_p(\mathcal{G}_n)\}_{n \geq 1}$ are non-increasing sequences of positive reals.

For each $n \geq 1$, since $E(\mathcal{H}) = \cup E(\mathcal{C}_n) = \cup E(\mathcal{G}_n)$, there exists $n_1$ so that $\mathcal{C}_n \subseteq \mathcal{G}_{n_1}$ and $n_2$ so that $\mathcal{G}_n \subseteq \mathcal{C}_{n_2}$. Thus, by Lemma~\ref{lem:sub-hyp},
\[
 \lim_{n \to \infty} \id_p(\mathcal{G}_n) \leq \lim_{n \to \infty} \id_p(\mathcal{C}_n) \leq \lim_{n \to \infty} \id_p(\mathcal{G}_n),
\]
which completes the proof.
\end{proof}

As the limits of independence densities of a chain for a hypergraph do not depend on the particular choice of chain, these are used to define the independence density for a countable hypergraph.

\begin{definition}\label{def:inf-id}
Let $\mathcal{H}$ be a countable hypergraph and let $\{\mathcal{C}_n\}_{n \geq 1}$ be a chain for~$\mathcal{H}$. For every $p \in (0,1)$, the \emph{independence $p$-density of $\mathcal{H}$} is
\[
 \id_p(\mathcal{H}) = \lim_{n \to \infty} \id_p(\mathcal{C}_n).
\]
\end{definition}

Given a hypergraph $\mathcal{H}$ and a subset of vertices~$A$, the \emph{neighbourhood hypergraph} of $A$ is the set $\{F \subseteq A^c  \mid  F \cup A \in \mathcal{H}\}$, denoted $N_{\mathcal{H}}(A)$. The following lemma, whose proof is similar to that of Lemma~\ref{lem:match-bd}, describes the effect on the independence density of a hypergraph of adding an independent set, $X$, as a hyperedge in terms of the matching number of the neighbourhood hypergraph of the set $X$.

\begin{lemma}\label{lem:add-set}
Let $\mathcal{H}$ be a finite hypergraph with $r(\mathcal{H})= k$. Let $X \subseteq V(\mathcal{H})$ be an independent set and suppose that $\{Y_1, Y_2, \dots, Y_m\}$ are pairwise disjoint, non-empty sets of vertices, disjoint from~$X$, and with the property that for each $i \leq m$, $X \cup Y_i \in \mathcal{H}$. Then, $r(\mathcal{H} \cup \{X\}) = k$, and
\[
 \id_p(\mathcal{H} \cup \{X\}) < \id_p(\mathcal{H}) \leq \id_p(\mathcal{H} \cup \{X\}) + p^{|X|}\big(1 - p^{k - |X|} \big)^m.
\]
\end{lemma}
\begin{proof}
Set $\mathcal{H}_X = \mathcal{H} \cup \{X\}$.  Note that since $X \cup Y_1 \in \mathcal{H}$, $|X| \leq k$ and so $r(\mathcal{H}_X) = k$ also.

For every set $I$ that is an independent set in $\mathcal{H}$ with $X \nsubseteq I$, $I$ is also independent in~$\mathcal{H}_X$. Indeed, $\mathcal{I}(\mathcal{H})$ can be written as a disjoint union
\[
 \mathcal{I}(\mathcal{H}) = \mathcal{I}(\mathcal{H}_X) \cup \{I \mid X \subseteq I,\ I \text{ independent in } \mathcal{H}\}.
\]

If $X \subseteq I$ and $I$ is independent in $\mathcal{H}$, then for every $i \in [1,m]$ the set $I$ does not contain every vertex from each of the sets $Y_1, Y_2, \dots, Y_m$.  Since the sets $Y_i$ are pairwise disjoint,
\begin{align*}
\sum_{\underset{X \subseteq I}{I \in \mathcal{I}(\mathcal{H})}} p^{|I|}(1-p)^{|V(\mathcal{H})| - |I|}
	&\leq p^{|X|} \prod_{i = 1}^m\big(1-p^{|Y_i|}\big)\\
	&\leq p^{|X|} \big(1-p^{k-|X|}\big)^m.
	\end{align*}

Thus, since $X$ is independent in~$\mathcal{H}$, but not in~$\mathcal{H}_X$,
\[
 \id_p(\mathcal{H}_X) < \id_p(\mathcal{H}) \leq \id_p(\mathcal{H}_X) + p^{|X|}\big(1 - p^{k - |X|} \big)^m.
\]
\end{proof}

The following Corollary extends Lemma~\ref{lem:add-set} to infinite hypergraphs.

\begin{corollary}\label{cor:inf-match}
Let $\mathcal{H}$ be a countably infinite hypergraph with $r(\mathcal{H}) = k$ and let $X \notin \mathcal{H}$ be an independent set of vertices with the property that there is an infinite collection $\{Y_i\}_{i\geq 1}$ of pairwise disjoint sets of vertices, disjoint from $X$ such that for every $i \geq 1$, $X \cup Y_i \in \mathcal{H}$. Then, $r(\mathcal{H} \cup \{X\}) = r(\mathcal{H})$ and $\id_p(\mathcal{H}\cup \{X\}) = \id_p(\mathcal{H})$.
\end{corollary}
\begin{proof}
Let $\{\mathcal{C}_n\}_{n \geq 1}$ be any chain for $\mathcal{H}$ with the property that for each $n \geq 1$, $\{X \cup Y_i\}_{i = 1}^n \subseteq \mathcal{C}_n$. By Lemma~\ref{lem:add-set},
\[
 \id_p(\mathcal{C}_n \cup \{X\}) \leq \id_p(\mathcal{C}_n) \leq \id_p(\mathcal{C}_n \cup \{X\}) + p^{|X|}\big(1 - p^{k-|X|} \big)^n
\]
and since the sequence $\{\mathcal{C}_n \cup \{X\}\}_{n \geq 1}$ is a chain for $\mathcal{H} \cup \{X\}$,
\[
 \id_p(\mathcal{H} \cup \{X\}) = \lim_{n \to \infty} \id_p(\mathcal{C}_n \cup \{X\}) = \lim_{n \to \infty} \id_p(\mathcal{C}_n) = \id_p(\mathcal{H}).
\]
\end{proof}

Note that if $\mathcal{H}$ is a hypergraph with $A \subseteq B$ and $A, B \in \mathcal{H}$, then $\id_p(\mathcal{H}) = \id_p(\mathcal{H} \setminus \{B\})$ since any independent set in $\mathcal{H}$ does not contain all of the vertices of $A$ and hence does not contain the set~$B$. Thus, we are always free to assume that any hypergraph in question is an antichain since hyperedges containing other hyperedges can be discarded without changing the independence density. In particular, in the situation described in Corollary~\ref{cor:inf-match}, adding the set $X$ to $\mathcal{H}$ and then deleting all of the sets $\{X \cup Y_i\}_{i \geq 1}$ from the set of hyperedges does not change the independence density, though it may decrease the rank.  The fact that this can be done simultaneously for all such sets in a consistent way is described in the following lemma.

\begin{lemma}\label{lem:no-inf-match}
Let $\mathcal{H}$ be a countable hypergraph with $r(\mathcal{H}) = k$.  There exists a countable hypergraph $\mathcal{G}$ with $r(\mathcal{G}) \leq k$ with the property that for every $p \in (0,1)$, $\id_p(\mathcal{H}) = \id_p(\mathcal{G})$ and for every $A \subseteq V(\mathcal{G})$, the neighbourhood hypergraph of $A$ in $\mathcal{G}$ has no infinite matching.
\end{lemma}

\begin{proof}
Define the set of subsets of vertices
\[
\mathcal{A} = \{A \subseteq V(\mathcal{H}) \mid \mu(N_{\mathcal{H}}(A)) = \infty\}
\]
and set $\mathcal{H}' = \mathcal{H} \cup \mathcal{A}$.  In order to see that $\id_p(\mathcal{H}) = \id_p(\mathcal{H}')$, order the set $\mathcal{A} = \{A_1, A_2, \ldots \}$ in any way and let $\{\mathcal{C}_n\}_{n \geq 1}$ be a chain for $\mathcal{H}$  chosen so that for each $n \geq 1$, the vertex set of $\mathcal{C}_n$ contains $\cup_{i = 1}^n A_i$.  Further assume that if $i \in [1, n]$,  then $N_{\mathcal{C}_n}(A_i)$ contains a matching of size at least~$n$.   Define a chain for $\mathcal{H}'$ by setting $\mathcal{C}_n' = \mathcal{C}_n \cup \{A_1, A_2, \ldots, A_n\}$ for each $n \geq 1$.  Applying Lemma~\ref{lem:add-set} once for each set in $\{A_1, A_2, \ldots, A_n\}$ yields $\id_p(\mathcal{C}_n) \leq \id_p(\mathcal{C}_n') \leq \id_p(\mathcal{C}_n) + n p (1-p^{k-1})^n$.  Taking limits as $n$ tends to infinity gives $\id_p(\mathcal{H}) = \lim_{n \to \infty} \id_p(\mathcal{C}_n) = \lim_{n \to \infty} \id_p(\mathcal{C}_n') = \id_p(\mathcal{H}')$.

Next, define a new hypergraph by discarding all hyperedges in $\mathcal{H}'$ that contain some other hyperedge:
\[
\mathcal{G} = \{E \in \mathcal{H}' \mid \nexists\ B \subsetneq E \text{ with } B \in \mathcal{H}'\}.
\]
In order to show that $\id_p(\mathcal{G}) = \id_p(\mathcal{H}')$ and hence $\id_p(\mathcal{G}) = \id_p(\mathcal{H})$, let $\{\mathcal{D}_n\}_{n \geq 1}$ be a chain for $\mathcal{H}'$.  Note that since $\{\mathcal{D}_n\}_{n \geq 1}$ is a chain, then for every $n \geq 1$, if $B \in \mathcal{D}_n$ and $A \subseteq B$ with $A \in \mathcal{H}'$, we have $A \in \mathcal{D}_n$ also.  Define a chain for $\mathcal{G}$ by setting, for each $n \geq 1$, 
\[
\mathcal{D}_n' = \{E \in \mathcal{D}_n \mid \nexists\ B \subsetneq E \text{ with } B \in \mathcal{D}_n\}.
\]
For each $n \geq 1$, $\mathcal{D}_{n}'$ is an induced subhypergraph of $\mathcal{D}_{n+1}'$.  Further, since $\mathcal{D}_n'$ is a finite hypergraph that is obtained from $\mathcal{D}_n$ by deleting all hyperedges that contain some other hyperedge, by the comment following Corollary~\ref{cor:inf-match}, $\id_p(\mathcal{D}_n) = \id_p(\mathcal{D}_n')$ and $r(\mathcal{D}_n') \leq k$.  Taking limits gives $\id_p(\mathcal{G}) = \id_p(\mathcal{H}')$ and $r(\mathcal{G}) \leq k$.
  
  Thus, we have $\id_p(\mathcal{G}) = \id_p(\mathcal{H})$. What remains is to show that no set of vertices has an infinite matching in its neighbourhood hypergraph in~$\mathcal{G}$.

Fix $B \subseteq V(\mathcal{G})$ and suppose, in hopes of a contradiction, that $\mu(N_{\mathcal{G}}(B)) = \infty$. Let $\{X_i\}_{i \geq 1}$ be an infinite matching in $N_{\mathcal{G}}(B)$.  Note that if there are infinitely many $i$ with $B\cup X_i \in \mathcal{H}$, then $B \in \mathcal{A}$ and so by the definition of~$\mathcal{G}$, there would be no hyperedges in $\mathcal{G}$ that contain $B$ (except $B$ itself).

Thus, for all but finitely many $i$, $B\cup X_i \in \mathcal{A}$. These shall be used to construct an infinite matching in $N_{\mathcal{H}}(B)$, which is a contradiction. Fix $n_1$ so that $B \cup X_{n_1} \in \mathcal{A}$ and, by the definition of~$\mathcal{A}$, pick $F_{1} \in N_{\mathcal{H}}(B\cup X_{n_1})$.  Then, $F_{1} \cup X_{n_1} \in N_{\mathcal{H}}(B)$.

Proceeding recursively, suppose that $n_1, n_2, \dots, n_\ell$ and $F_1, F_2, \dots, F_\ell$ are given so that for every $i \in [1, \ell]$, $B\cup X_{n_i} \in \mathcal{A}$, $F_i \in N_{\mathcal{H}}(B\cup X_{n_i})$ and the set $\{F_i \cup X_{n_i}\}_{i = 1}^{\ell}$ is a matching of size $\ell$ in $N_{\mathcal{H}}(B)$. Since the set $V_\ell = \cup_{i = 1}^{\ell} (F_i \cup X_{n_i})$ is finite and the sets $\{X_i\}_{i \geq 1}$ are pairwise disjoint, there exists $n_{\ell+1} > n_{\ell}$ so that $X_{n_{\ell+1}} \cap V_\ell = \emptyset$ and $B\cup X_{n_{\ell+1}}\in \mathcal{A}$. Then, since $N_{\mathcal{H}}(B\cup X_{n_{\ell+1}})$ contains an infinite matching, there is some $F_{\ell+1} \in N_{\mathcal{H}}(B\cup X_{n_{\ell+1}})$ with $F_{\ell+1} \cap V_\ell = \emptyset$.  Adding $F_{\ell+1}\cup X_{n_{\ell+1}}$ to $V_\ell$ yields a matching of size $\ell+1$ in $N_{\mathcal{H}}(B)$.  Thus, since this sequence of matchings is increasing, $N_{\mathcal{H}}(B)$ contains an infinite matching, which is a contradiction.
\end{proof}

\section{Sequences of independence densities}\label{sec:results}

In this section, the proof of Theorem~\ref{thm:closed-finite} is given. Using the same tools, it is shown that for every $k \geq 2$, the set $\mathbf{H}_{k,p}$ contains no infinite increasing sequences.

The key tool to prove both these results is Proposition~\ref{prop:almost-const-subseq} below, which describes how one can assume that sequences in $\mathbf{H}_{k,p}$ that are bounded away from $0$ arise from hypergraphs sharing a common finite subhypergraph with sets outside this finite `core' associated with large matchings in certain neighbourhood hypergraphs.

\begin{proposition}\label{prop:almost-const-subseq}
Let $k \geq 2$ and let $\{x_n \mid n \geq 1 \} \subseteq \mathbf{H}_{k,p}$  with $\inf_{n \geq 1} x_n = x > 0$. There exists a sequence of countable hypergraphs $\{\mathcal{H}_n\}_{n \geq 1}$ such that for every~$n$, $r(\mathcal{H}_n) \leq k$, $\id_p(\mathcal{H}_n) = x_n$, and a finite hypergraph $\mathcal{H}_0$ on vertex set $V_0$ and an increasing sequence $\{n_i\}_{i \geq 1}$ such that for every $i \geq 1$,
\begin{enumerate}[(a)]
	\item $\mathcal{H}_{n_i}[V_0] = \mathcal{H}_0$, and
	\item if $E \in \mathcal{H}_{n_i}$ with $E \setminus V_0 \neq \emptyset$, then there exists $A \subseteq E \cap V_0$ with
	\[
	 \mu\big(\{F \subseteq V_0^c \mid F \cup A \in \mathcal{H}_{n_i}\}\big) \geq i.
	\]
\end{enumerate}
\end{proposition}
\begin{proof}
Let $\{\mathcal{H}_n\}_{n \geq 1}$ be a sequence of countable hypergraphs with the property that for every~$n$, $\id_p(\mathcal{H}_n) = x_n$ and $r(\mathcal{H}_n) \leq k$.  Assume throughout, by Lemma~\ref{lem:no-inf-match}, that for every set $A$ and every~$n$, $\mu(\{F \mid F \cap A = \emptyset \text{ and } F \cup A \in \mathcal{H}_n\}) < \infty$.

The proof follows by a recursive construction. For every $j \leq k$, we construct a finite hypergraph $\mathcal{H}^{j}$ on vertex set $V_j$ and a subsequence $\{n_i^{(j)}\}_{i \geq 1}$ such that for every $i \geq 1$, with an appropriate relabelling of the vertices of $\mathcal{H}_{n_i^{(j)}}$,
\begin{enumerate}[(a)]
	\item $\mathcal{H}_{n_i^{(j)}}[V_j] = \mathcal{H}^j$, and
	\item if $E \in \mathcal{H}_{n_i^j}$ with $E \setminus V_j \neq \emptyset$, then either
		\begin{itemize}
			\item $\exists\ A \subseteq E \cap V_j$ with $\mu(\{F \subseteq V_j^c \mid F\cup A \in \mathcal{H}_{n_i^{(j)}}\}) \geq i$, or
			\item $|E \cap V_j| \geq j$.
		\end{itemize}
\end{enumerate}
The result then follows from the case $j = k$ since the hypergraphs $\mathcal{H}_{n_i^{(k)}}$ are of rank at most $k$ and so there is no hyperedge $E$ with $E \setminus V_k \neq \emptyset$ and $|E \cap V_k| \geq k$.

To begin the recursive construction with $j = 1$, note that since $\inf_{n} \id_p(\mathcal{H}_n) = x > 0$, then, by Lemma~\ref{lem:match-bd}, for every~$n$, $\mu(\mathcal{H}_n)$ is finite and
\[
 \mu(\mathcal{H}_n) \leq m_1 = \left\lfloor \frac{\log x}{\log(1 - p^k)} \right\rfloor.
\]

Since there are only finitely many rank $k$ hypergraphs on at most $m_1 k$ vertices, there is an infinite increasing sequence $\{n_i^{(1)}\}_{i \geq 1}$ and a finite hypergraph $\mathcal{H}^1$ on vertex set $V_1$ such that for every $i \geq 1$, after possibly relabelling vertices, $\mathcal{H}_{n_i^{(1)}}[V_1] = \mathcal{H}^1$ and a maximum matching of $\mathcal{H}_{n_i^{(1)}}$ is contained in~$\mathcal{H}^1$.

Since $\mathcal{H}^1$ contains a maximum matching for each~$\mathcal{H}_{n_i^{(1)}}$, if $E \setminus V_1 \neq \emptyset$ and $E \in \mathcal{H}_{n_i^{(1)}}$, then $E \cap V_1 \neq \emptyset$ and so $|E \cap V_1 | \geq 1$. This completes the base case of the recursion.

For the recursion step, suppose that $j \leq k-1$ and that $V_j$, $\mathcal{H}^j$, and $\{n_i^{(j)}\}_{i \geq 1}$ are given with the desired properties. Since $V_j$ is a finite set, there is an increasing sequence $\{{\ell}_i^{(j)}\}_{i \geq 1}$ which is a subsequence of $\{n_i^{(j)}\}$ with the property that if $A \subseteq V_j$ with
\[
 \sup_i \left\{\mu\big(\{F \subseteq V_j^c \mid F \cup A \in \mathcal{H}_{\ell_i^{(j)}}\}\big)\right\} = \infty,
\]
then for every $i \geq 1$, $\mu(\{F \subseteq V_j^c \mid F \cup A \in \mathcal{H}_{\ell_i^{(j)}}\}) \geq i$. Set
\[
 \mathcal{M}_j =\left\{A \subseteq V_j \mid \sup_i \left\{ \mu\big(\{F \subseteq V_j^c \mid F\cup A \in \mathcal{H}_{\ell_i^{(j)}}\}\big)\right\} < \infty \right\}.
\]
The set $\mathcal{M}_j$ is the set of all subsets of $V_j$ whose neighbourhood hypergraphs in the sequence have bounded matching number. Define the largest such matching number to be
\[
 m_j = \max_{A \in \mathcal{M}_j} \sup_i \left\{ \mu\big(\{F \subseteq V_j^c \mid F\cup A \in \mathcal{H}_{\ell_i^{(j)}}\}\big)\right\}.
\]
Then, $m_j < \infty$ by the definition of $\mathcal{M}_j$ and since $V_j$ is finite.

Since $\mathcal{M}_j$ is finite and there are finitely many rank at most $k$ hypergraphs on at most $k m_j$ vertices, there is a collection of finite hypergraphs $\{\mathcal{G}_A\}_{A \in \mathcal{M}_j}$ on the vertices of $V_j^c$ and an increasing sequence $\{m_i^{(j+1)}\}_{i \geq 1}$ that is a subsequence of $\{\ell_i^{(j)}\}_{i \geq 1}$ so that for every $A \in \mathcal{M}_j$ and for every $i$, possibly relabelling vertices in $V_j^c$,
\[
 \{F \subseteq V(\mathcal{G}_A) \mid F \cup A \in \mathcal{H}_{m_i^{(j+1)}}\} = \mathcal{G}_A
\]
and so that a maximum matching of $\{F \subseteq V_j^c \mid F \cup A \in \mathcal{H}_{m_i^{(j+1)}}\}$ is contained in~$\mathcal{G}_A$.

Set $V_{j+1} = V_j \cup \left( \cup_{A \in \mathcal{M}_j} V(\mathcal{G}_A) \right)$. Since $V_{j+1}$ is finite, there exists $\{n_i^{(j+1)}\}_{i \geq 1}$ that is a subsequence of $\{m_{i}^{(j+1)}\}_{i \geq 1}$ and a finite hypergraph $\mathcal{H}^{j+1}$ on $V_{j+1}$ so that for every $i$, $\mathcal{H}_{n_i^{(j+1)}}[V_{j+1}] = \mathcal{H}^{j+1}$.  By passing to a further subsequence, we can further assume that if $A \subseteq V_{j+1}$ is such that
\[
 \sup_i \left\{\mu\big(\{F \subseteq V_{j+1}^c \mid F \cup A \in \mathcal{H}_{n_{i}^{(j+1)}}\}\big)\right\} = \infty,
\]
then, for every $i \geq 1$, $\mu\big(\{F \subseteq V_{j+1}^c \mid F \cup A \in \mathcal{H}_{n_{i}^{(j+1)}}\}\big) \geq i$.

For the second condition still to prove, let $E \in \mathcal{H}_{n_i^{(j+1)}}$ be such that $E \setminus V_{j+1} \neq \emptyset$ and suppose that for every $A \subseteq E \cap V_{j+1}$,
\[
 \sup_i \left\{ \mu\big(\{F \subseteq V_{j+1}^c \mid F\cup A \in \mathcal{H}_{n_i^{(j+1)}}\}\big)\right\} < \infty.
\]
In particular, since any matching in the neighbourhood hypergraph of $E \cap V_j$ can contain at most $|V_{j+1} \setminus V_j|$ sets intersecting $V_{j+1} \setminus V_j$, then the matchings in the neighbourhood hypergraph of $E$ in the sequence of hypergraphs $\{\mathcal{H}_{n_i^{j}}\}_{i \geq 1}$ were bounded.  Thus, by the induction hypothesis, $|E \cap V_j| \geq j$ and since $V_{j+1}$ contains a maximum matching for the hypergraphs
\[
 \{F \subseteq V_j^c \mid F \cup (E \cap V_{j}) \in \mathcal{H}_{n_i^{(j+1)}}\},
\]
and there is at least one such edge, then $E \cap (V_{j+1} \setminus V_j) \neq \emptyset$ and so $|E \cap V_{j+1}| = |E \cap V_j| + |E \cap (V_{j+1} \setminus V_j)| \geq j+1$.

This completes the recursion step and hence the proof.
\end{proof}

For convenience, we now restate Theorem~\ref{thm:closed-finite} and give its proof.

\begin{mainthm}
For every $k \geq 1$ and $p \in (0,1)$, the set $\mathbf{H}_{k,p}$ is closed and furthermore,
\[
 \mathbf{H}_{k,p} = \{0\} \cup \{\id_p(\mathcal{H}) \mid \mathcal{H} \text{ is a finite hypergraph with } r(\mathcal{H}) \leq k\}.
\]
\end{mainthm}
\begin{proof}
The hypergraph that consists of a countably infinite matching of sets of size $k$ has independence density $0$ and so we now consider only sequences of independence densities converging to positive numbers.

Let $\{\mathcal{H}_n\}_{n \geq 1}$ be a sequence of hypergraphs with rank at most $k$ and with $\lim_{n \to \infty} \id_p(\mathcal{H}_n) = x > 0$. By Proposition~\ref{prop:almost-const-subseq}, assume without loss of generality that there exists a finite set $V_0$ and a finite hypergraph $\mathcal{H}_0$ on $V_0$ so that for every $n \geq 1$, $\mathcal{H}_n[V_0] = \mathcal{H}_0$, and for every $E \in \mathcal{H}_n$ with $E \setminus V_0 \neq \emptyset$ there exists $A \subseteq E \cap V_0$ such that for every~$n$,
\[
 \mu\big(\{F \subseteq V_0^c \mid F \cup A \in \mathcal{H}_n\}\big) \geq n.
\]

Let $\mathcal{A}$ be the collection of all such sets $A \subseteq V_0$ whose neighbourhood hypergraphs contain unbounded matchings in $V_0^c$ and set $\mathcal{H}' = \mathcal{H}_0 \cup \mathcal{A}$. The hypergraph $\mathcal{H}'$ is finite and the remainder of the proof consists of showing that $\id_p(\mathcal{H}') = x$.

By the choice of $\mathcal{A}$ and applying Lemma~\ref{lem:add-set} repeatedly to finite hypergraphs in appropriately chosen chains for $\mathcal{H}_n$ and $\mathcal{H}_n \cup \mathcal{A}$,
\begin{align*}
\id_p(\mathcal{H}_n \cup \mathcal{A})  \leq \id_p(\mathcal{H}_n) & \leq \id_p(\mathcal{H}_n \cup \mathcal{A}) + \sum_{A \in \mathcal{A}} p^{|A|}\big(1 -p^{k - |A|} \big)^n\\
	&\leq \id_p(\mathcal{H}_n \cup \mathcal{A}) + (1+p)^{|V_0|} \big(1 - p^k \big)^n.
\end{align*}
Thus,
\[
 \lim_{n \to \infty} \id_p(\mathcal{H}_n \cup \mathcal{A}) = \lim_{n \to \infty} \id_p(\mathcal{H}_n) = x.
\]
Now, for every hyperedge $E \in \mathcal{H}_n$ with $E \nsubseteq V_0$, there exists $A \subseteq E$ with $A \in \mathcal{A}$. Thus, for every~$n$, $\id_p(\mathcal{H}_n \cup \mathcal{A})  = \id_p(\mathcal{H}')$.

Therefore, $\mathcal{H}'$ is the desired finite hypergraph with $r(\mathcal{H}') \leq k$ and $\id_p(\mathcal{H}') = x$.
\end{proof}

In the special case $k = 2$, Theorem~\ref{thm:closed-finite} guarantees that for every countable graph, there is a finite hypergraph of rank at most $2$ with the same independence density.  In fact, in the following proposition, it is shown that when $p = 1/2$, there is a finite graph with the same independence density.  This is the only place in this paper where a particular value for $p \in (0,1)$ is required.

\begin{proposition}\label{prop:graphs}
Let $G$ be a graph on a countable vertex set with $\id(G) > 0$. Then there is a finite graph $H$ with $\id(G) = \id(H)$.
\end{proposition}
\begin{proof}
By Theorem~\ref{thm:closed-finite}, there is a finite rank at most $2$ hypergraph $\mathcal{H}$ with $\id(G) = \id(\mathcal{H})$. If every hyperedge in $\mathcal{H}$ has exactly $2$ vertices, then $\mathcal{H}$ is the desired finite graph. Otherwise, set $B = \{x \in V(\mathcal{H}) \mid \{x \} \in \mathcal{H}\}$ and let $V' = \cup_{x \in B}\{a(x), b(x), c(x)\}$ be a collection of $3|B|$ new vertices. If there are any pairs in $\mathcal{H}$ containing vertices of~$B$, these can be deleted without changing the independence density.

Define a graph $H$ on $(V(\mathcal{H}) \cup V') \setminus B$ with edges given by those pairs in $\mathcal{H}$ not containing vertices from $B$ and for every $x \in B$, a copy of $K_3$ on the vertices $\{a(x), b(x), c(x)\}$.

Note that if $I$ is independent in~$\mathcal{H}$, then $I \cap B = \emptyset$ and since $i(K_3) = 4$,
\[
 i(H) = i(\mathcal{H}) \cdot \prod_{x \in B} i(K_3) = i(\mathcal{H}) 4^{|B|}.
\]
Thus,
\[
 \id(H) = \frac{i(\mathcal{H} )\cdot 4^{|B|}}{2^{|V(\mathcal{H})| - |B| + 3|B|}} = \frac{i(\mathcal{H})}{2^{|V(\mathcal{H})|}} \cdot \frac{4^{|B|}}{2^{2|B|}} = \id(\mathcal{H}) = \id(G)
\]
by the choice of~$\mathcal{H}$.
\end{proof}

For arbitrary $p$, something like Proposition~\ref{prop:graphs} need not hold.  As an example, if $G = K_{1, \infty}$, an infinite star, then $\id_p(G) = 1-p$.  If $p$ is transcendental, there is no finite graph with $\id_p(G) = 1-p$.

A further consequence of Proposition~\ref{prop:almost-const-subseq}, is that the set $\mathbf{H}_{k,p}$ is `reverse well-ordered'.

\begin{theorem}\label{thm:no-inc-seq}
The set $\mathbf{H}_{k,p}$ contains no infinite increasing sequences.
\end{theorem}
\begin{proof}
Let $\{\mathcal{H}_n\}_{n \geq 1}$ be a sequence of countable hypergraphs with $r(\mathcal{H}_n) \leq k$ and non-decreasing independence densities:
\[
 0 < \id_p(\mathcal{H}_1) \leq \id_p(\mathcal{H}_2) \leq \dots
\]
It shall be shown that this sequence is eventually constant.  Indeed, if $\{\id_p(\mathcal{H}_n)\}$ contains a subsequence that is eventually constant, then the original sequence is itself eventually constant.

By Theorem~\ref{thm:closed-finite}, assume that all of the hypergraphs $\mathcal{H}_n$ are finite and again by Proposition~\ref{prop:almost-const-subseq}, possibly passing to a subsequence, assume that there is a finite hypergraph $\mathcal{H}_0$ on vertex set $V_0$ so that for every~$n$, $\mathcal{H}_n[V_0] = \mathcal{H}_0$ and that if $E \in \mathcal{H}_n$ with $E \setminus V_0 \neq \emptyset$, then there exists $A \subseteq E \cap V_0$ such that for every $n \geq 1$,
\[
 \mu\big(\{F \subseteq V_0^c \mid F \cup A \in \mathcal{H}_n \}\big) \geq n.
\]
Let $\mathcal{A}$ be the collection of all such sets $A \subseteq V_0$. As before, assume without loss of generality that each $\mathcal{H}_n$ is an antichain.

If $\mathcal{A} = \emptyset$, then $\mathcal{H}_n = \mathcal{H}_0$ and so the sequence $\{\id_p(\mathcal{H}_n)\}_{n \geq 1}$ is constant.

Suppose now, in hopes of a contradiction, that $\mathcal{A} \neq \emptyset$. Note that, by the definition of $\mathcal{A}$ and then since $\mathcal{A} \neq \emptyset$, for every~$n$,
\[
 \id_p(\mathcal{H}_0 \cup \mathcal{A}) < \id_p(\mathcal{H}_n).
\]
On the other hand, as in the proof of Theorem~\ref{thm:closed-finite},
\[
 \id_p(\mathcal{H}_0 \cup \mathcal{A}) = \lim_{n \to \infty} \id_p(\mathcal{H}_n).
\]
Then
\begin{align*}
\lim_{n \to \infty} \id_p(\mathcal{H}_n)
	& = \id_p(\mathcal{H}_0 \cup \mathcal{A})\\
	& < \id_p(\mathcal{H}_1)\\
	& \leq \lim_{n \to \infty} \id_p(\mathcal{H}_n)	&&\text{(since the sequence is non-decreasing)}
\end{align*}
which is a contradiction.

Thus, $\mathcal{A} = \emptyset$ and for all~$n$, $\id_p(\mathcal{H}_n) = \id_p(\mathcal{H}_0)$.
\end{proof}

\section{Open Problems}\label{sec:open}

In this section, we note some open problems related to the results given in this paper.

In Theorem~\ref{thm:no-inc-seq}, it was shown that for any $k, p$, the set $\mathbf{H}_{k,p}$ does not contain any infinite increasing sequences though it does contain infinite decreasing sequences.  Thus, this set of reals, with the usual ordering reversed, has the order type of some infinite ordinal number.

\begin{question}
For any $k \geq 2$, $p \in (0,1)$, what is the order type of the set $\mathbf{H}_{k,p}$?
\end{question}

In the proof of Theorem~\ref{thm:closed-finite}, hyperedges are added to the hypergraph that are smaller than some fixed hyperedge already in the hypergraph.  What can be said if we are restricted to uniform hypergraphs?

\begin{question}
In the case $p = 1/2$ can Proposition~\ref{prop:graphs} be extended to $k > 2$?  That is, for any $k \geq 2$, if $\mathcal{H}$ is a $k$-uniform countable hypergraph with $\id(\mathcal{H}) > 0$, does there exist a finite $k$-uniform hypergraph $\mathcal{H}'$ with $\id(\mathcal{H}') = \id(\mathcal{H})$?
\end{question}

\end{document}